\documentclass[12pt,a4paper,reqno]{amsart}
\usepackage{a4wide}
\usepackage{amssymb,amsmath,amsthm,mathrsfs,accents}
\usepackage{graphicx}
\usepackage{colortbl}
\usepackage{enumerate}
\usepackage{upref}
\usepackage[bookmarks=false]{hyperref}
\usepackage{mathtools}
\usepackage{caption}
\usepackage{floatrow}
\usepackage{lmodern}

\theoremstyle{plain}
\newtheorem{maintheorem}{Theorem}

\newtheorem{maincorollary}[maintheorem]{Corollary}
\newtheorem{theorem}{Theorem}[section]

\newtheorem{proposition}[theorem]{Proposition}

\theoremstyle{definition}
\newtheorem{definition}{Definition}[section]

\newtheorem*{condition}{Condition}

\theoremstyle{remark}
\newtheorem{remark}{Remark}[section]

\usepackage{enumitem}
\makeatletter
\def\namedlabel#1#2{\begingroup
   #2%
 \def\@currentlabel{#2}%
   \phantomsection\label{#1}\endgroup
}

\def\nn{\ensuremath{\mathscr N}}
\def\R{\ensuremath{\mathbb R}}
\def\N{\ensuremath{\mathbb N}}

\def\I{\ensuremath{{\bf 1}}}
\def\e{{\ensuremath{\rm e}}}

\def\G{\ensuremath{\mathcal G}}
\def\RR{\ensuremath{\mathcal R}}

\def\B{\ensuremath{\mathcal B}}
\def\M{\ensuremath{\mathcal M}}

\def\l{{\rm Leb}}
\def\P{\ensuremath{\mathcal P}}
\def\p{\ensuremath{\mathbb P}}
\def\F{\ensuremath{\mathcal F}}

\def\t{\ensuremath{t}}

\def\X{\mathcal{X}}
\def\ie{{\em i.e.}, }

\def\cv{\ensuremath{\text {Cor}}}

\def\e{\rm e}
\def\o{\ensuremath{\underline{\omega}}}
\newcommand{\dif}{\mathrm{d}}

\newcommand{\un}{\underline}
\newcommand{\ov}{\overline}

\def\dist{\ensuremath{\text{dist}}}

\def\supp{\ensuremath{\text{supp}}}

\def\sm{\setminus}
\def\eps{\varepsilon}
\def\vfi{\varphi}
\def\wt{\widetilde}

\newcommand{\qand}{\quad\text{and}\quad}

\def\sm{\ensuremath{\mu_\varepsilon}}

\mathchardef\ordinarycolon\mathcode`\:
\mathcode`\:=\string"8000
\begingroup \catcode`\:=\active
  \gdef:{\mathrel{\mathop\ordinarycolon}}
\endgroup

\numberwithin{equation}{section}

\begin{document}

\title[Laws of rare events for transitive random maps]{Decay of
  correlations and laws of rare events for transitive random maps}

\author[V. Ara\'{u}jo]{V\'itor Ara\'{u}jo}

\address{Vitor Ara\'ujo,
 Departamento de Matem\'atica, Universidade Federal da Bahia\\
Av. Ademar de Barros s/n, 40170-110 Salvador, Brazil.}
\email{vitor.d.araujo@ufba.br,
  www.sd.mat.ufba.br/$\sim$vitor.d.araujo}
\email{vitor.araujo.im.ufba@gmail.com}

\author[H. Ayta\c{c}]{Hale Ayta\c{c}}

\address{Hale Ayta\c{c},
 Departamento de Matem\'atica, Universidade Federal da Bahia\\
Av. Ademar de Barros s/n, 40170-110 Salvador, Brazil.}
\email{aytachale@gmail.com}

\date{\today}



\begin{abstract}
  We show that a uniformly continuous random perturbation of
  a transitive map defines an aperiodic Harris chain which
  also satisfies Doeblin's condition. As a result, we get
  exponential decay of correlations for suitable random
  perturbations of such systems.  We also prove that, for
  transitive maps, the limiting distribution for Extreme
  Value Laws (EVLs) and Hitting/Return Time Statistics
  (HTS/RTS) is standard exponential. Moreover, we show that
  the Rare Event Point Process (REPP) converges in
  distribution to a standard Poisson process.
  \end{abstract}

\maketitle
\tableofcontents

\section{Introduction}

Deterministic discrete dynamical systems are often used to
model physical phenomena. However, it is more realistic to
consider random perturbations of such systems to take into
consideration the observational errors. The behaviour of such
random systems has been studied thoroughly in the last
decades. We mention, for example, \cite{Ki86,KiLi06} for
excellent expositions on the subject.

Laws of rare events for chaotic (deterministic) dynamical
systems have also been extensively studied in the last
years. By \textit{rare events} we mean that the probability
of the event is small. In the literature, these notions were
first described as Hitting/Return Times Statistics
(HTS/RTS). In this setting, rare events correspond to
entrances in small regions of the phase space and the goal
is to prove distributional limiting laws for the normalised
waiting times before hitting/returning to these
asymptotically small sets.  For a nice review on the
subject we mention \cite{Sau09}. More recently, rare events
have also been studied through Extreme Value Laws (EVLs). In
this setting, rare events correspond to exceedances of a
high level and one looks for the distributional limit of the
partial maxima of stochastic processes arising from such
chaotic systems simply by evaluating an observable function
along the orbits of the system. We refer the reader to
\cite{Freitas13} for a review on this subject. It turns out that
these are just two views on the same phenomena: there is a
link between these two approaches. This link was already
perceivable in the pioneering work of Collet, \cite{collet2001},
however it has been formally proved in \cite{FFT10,FFT11},
where Freitas et al.\ showed that under general conditions
on the observable functions, the existence of HTS/RTS is
equivalent to the existence of EVLs. These observable
functions achieve a maximum (possibly $\infty$) at some
chosen point $\zeta$ in the phase space so that the rare
event of  an exceedance of a high level occurring corresponds
to an entrance in a small ball around $\zeta$. Moreover, the
study of rare events may be enhanced if we enrich the
process by considering multiple exceedances (or hits/returns
to target sets) that are recorded by Rare Events Point
Processes (REPP), which count the number of exceedances (or
hits/returns) in a certain time frame. Then the aim is to
get limits in distribution for such REPP when time is
adequately normalised.

Very recently, the connection between EVLs and HTS/RTS for
deterministic dynamics was extended to the random case
through additive random perturbations in Aytaç, Freitas and
Vaienti \cite{AFV15}.  There additive random perturbation of
expanding and piecewise expanding maps (with finite
branches) was studied and a standard exponential law was
obtained. Since then several advances have been obtained in
this direction (we will mention some of them below). We
remark that it was in the random setting that the
fundamental theory of extreme value was developed and there
are two approaches for studying the recurrence properties of
the underlying system. In \cite{MR11}, Marie and Rousseau
defined, for the first time, annealed and quenched return
times for systems generated by the composition of random
maps.  On the one hand, in the annealed approach, the
realisation is fixed and then integrated over all possible
realisations to get the law. In this case the product
measure for the skew-product system is studied. On the other
hand, to study the quenched approach we take a random
realisation, consider sample-stationary measures and get
limit laws for almost every realisation.

In \cite{FFLTV13}, Faranda et al.\ studied the additive
random perturbation of rational and irrational rotations and
proved, using the annealed approach, the existence of
extreme value laws for perturbed dynamics, regardless of the
intensity of the noise whereas there is no limiting law in
deterministic case.

In \cite{RSV14}, Rousseau et al.\ got an exponential law for
random subshifts of finite type. They showed that for
invariant measures with super-polynomial decay of
correlations hitting times to dynamically defined cylinders
satisfy an exponential distribution. They also got similar
results for random expanding maps.  Their results were
quenched exponential law for hitting times.

In \cite{Rous14}, Rousseau studied hitting and return time
statistics for observations of dynamical systems and got an
annealed exponential law for super-polynomially mixing
random dynamical systems.  This theory was applied to random
expanding maps, random circle maps expanding on average and
randomly perturbed dynamical systems.

Again in \cite{RoTo15}, Rousseau and Todd proved the existence
of quenched laws of hitting time statistics for random
subshifts of finite type. They showed that it was still
possible to get a dichotomy of standard vs non-standard
exponential laws for non-periodic and for periodic points
respectively, even in the random setting.


One of the main achievements of this paper is the
generalisations of the results in \cite{AFV15}. As was
pointed out in \cite{AFV15}, decay of correlations against
all $L^1$ observables was one of the main ingredients in the
theory. Hence, the results there were restricted to systems with
summable decay of correlations against all $L^1$ observables
with some additional conditions on the map. Here, we show
that random perturbations of any transitive dynamical system
(that is, admitting a forward dense orbit) defines an
aperiodic Harris chain which also satisfies Doeblin's
condition (see Proposition~\ref{prop:ran-per-aper-Harris})
which gives rise to a uniformly ergodic Markov chain.  It is
well-known that random perturbations are a special case of
Markov Chains with suitable transition
probabilities. Recently Jost et al.\ in \cite{JKR15} showed
that Markov Chains with regular enough transition densities
can be represented by continuous random maps or random
diffeomorphisms.  Then, using the known results for such
chains, we conclude that every transitive dynamical system
under uniformly continuous random perturbation has
exponential decay of correlations against all $L^1$
observables (see Theorem~\ref{mthm:trans->DC}). With this
approach, we get laws of rare events for a larger set of
dynamics, namely transitive systems, under more general
random perturbations.

Our work shows that EVL and HTS/RTS for a large class of
randomly perturbed dynamics is an application of the theory
of Stochastic Processes/Markov Chains needing very little
deterministic dynamical assumptions on the underlying
unperturbed system: we only need transitiveness and very
general random perturbations.


\section{Definitions and statement of Results}
\label{sec:statement-results}

Let $(\M,\B,\upsilon,f)$ be a discrete time deterministic
dynamical system, where $\M$ is a compact connected
finite dimensional Riemannian manifold;
$\dist(\cdot,\cdot)$ denotes the induced Riemannian distance
on $\M$ and $\l$ a normalised volume form on the
$\sigma$-algebra $\B$ of Borel sets of $\M$ that we call
Lebesgue measure; $f:\M\to\M$ is a measurable map, and $\upsilon$
is an $f$-invariant probability measure.

Consider the time series $X_0,X_1,X_2,\dots$ arising from
such a system simply by evaluating a given random variable
(r.v.) $\varphi:\M\to\R\cup\{+\infty\}$ along the orbits of
the system:
\begin{equation}
\label{eq:def-stat-stoch-proc-DS} X_n=\varphi\circ f^n,\quad \mbox{for
each } n\in {\mathbb N}.
\end{equation}

Clearly, $X_0,X_1,\dots$ defined in this way is not an
independent sequence.  However, invariance of $\upsilon$
guarantees that the stochastic process is stationary.

\subsection{Random perturbations. Representation of Markov chains}
\label{sec:random-perturbations}

We now consider a random setting constructed from the
deterministic system via perturbing the original map.  Let
$F:\M\times X\to\M$ be a parameterized family of measurable
maps
$f_\omega:\M\to\M, f_\omega(x):=F(x,\omega), \omega\in X,
x\in\M$,
where $(X,d)$ is a compact metric space.  We denote the ball
of radius $\varepsilon>0$ around $x\in \M$ by
$B_\eps(x):= \{y\in\M: \dist(x,y)<\eps\}$ and around
$\omega\in X$ by
$V_\eps(\omega):=\{\eta\in X: d(\eta,\omega)<\eps\}$.  For a
fixed $\omega^*\in X$ which we denote by $0$ and some
$\eps_0>0$, let $\theta=\theta_{\eps_0}$ be a Borel probability
measure so that $\supp(\theta)\supset V_{\eps_0}(0)$.

 We define a random perturbation of $f:\M\to\M$ by the pair
 $(f_\omega,\theta)$ which we assume satisfies
 \begin{equation}\label{eq:random-perturbation}
   f_0=f 
   \qand 
   f^x\big(\supp(\theta)\big)\supset
   B_{\rho_0}(fx), \quad\l-\text{a.e.\  }  x\in\M
 \end{equation}
 for a constant $\rho_0>0$, where we write
 $f^x(\omega):=f_{\omega}(x)$; and also
\begin{equation}
 \label{eq:noise-distribution}
(f^x_*) \theta=q_x\l
 \quad\text{with}\quad
\un{q}
\leq
q_x
\le
\ov{q}, \quad \l-\text{a.e.\ on } \supp(f^x_*\theta)
\end{equation}
for some constants $\ov{q}>\un{q}>0$ and $\l$-a.e.\ $x$.

 Consider a sequence of i.i.d.\ random variables (r.v.)~
 $W_1, W_2,\ldots$ taking values on $V_\eps(0)$, where
 $\omega^*=0$, with common distribution given by $\theta$.
 Let $\Omega=V_\eps(0)^\N$ denote the space of realisations
 of such processes and $\theta^\N$ the product measure defined
 on its Borel subsets. Given a point $x\in\M$ and the
 realisation of the stochastic process
 $\o=(\omega_1,\omega_2,\ldots)\in\Omega$, we define the
 random orbit of $x$ as $x, f_{\o}(x), f^2_{\o}(x),\ldots$
 where the evolution of $x$, up to time $n\in\N$, is
 obtained by the concatenation of the respective randomly
 perturbed maps
\begin{align*}
f_{\o}^n(x):=
f_{\omega_n}\circ f_{\omega_{n-1}}\circ\cdots\circ f_{\omega_1}(x),
\end{align*}
with $f_{\o}^0$ being the identity map on $\M$.

In our setting, the random perturbations we consider satisfy
the conditions expressed in the relations
\eqref{eq:random-perturbation} and
(\ref{eq:noise-distribution}), which can be said to be
\emph{uniformly continuous random perturbations} requiring
the small noise to uniformly cover a ball of positive radius
around the unperturbed transformation. Weaker assumptions
might be sufficient to obtain the same results, but we did
not search for the most general possible conditions.

\subsubsection{Existence of uniformly continuous random
  perturbations}
\label{sec:existence-uniformly-}

Families of random maps satisfying
\eqref{eq:random-perturbation} and
(\ref{eq:noise-distribution}) can be constructed from any
$C^r$ map on a compact finite $n$-dimensional manifold, as
showed in \cite[Example 2]{Ar00}, which we present below for
completeness. Here we can have $r\ge0$ comprising (H\"older)
continous or smooth maps; measurable maps are also allowed.

We start by taking a finite number of coordinate charts
$\{ \psi_i: B(0,3) \mapsto\M\}_{i=1}^l$ such that
$\{ \psi_i( B(0,3)) \}_{i=1}^l$ is an open cover of $\M$ and
$\{ \psi_i( B(0,1) ) \}_{i=1}^l$ also (this is a standard
construction, cf.~\cite[Sec. 1.2]{PM82}), where $B(0,a)$
denotes the ball of radius $a>0$ in $\R^n$.  In each of
those charts we define $n$ orthonormal vector fields
$\widetilde{X}_{i1},\ldots,\widetilde{X}_{in}: B(0,3)
\mapsto T_{\psi_i(B(0,3))}\M$
and extend them to the whole of $\M$ with the help of bump
functions.  This may be done in such a way that the
extensions $X_{ij}$ are null outside $\psi_i(B(0,2))$ and
coincide with $\widetilde{X}_{ij}$ in
$\psi_i(\overline{B(0,1)})$, $i=1,\ldots,l$; $j=1,\ldots,n$.
We then see that, since $\{\psi_i(B(0,1))\}_{i=1}^l$ was an
open cover of $\M$, \emph{at every $x\in M$ there is some
  $1\le i\le l$ such that $X_{i1}(x),\ldots,X_{in}(x)$ is an
  orthonormal basis for $T_x M$}.

We define the following parameterized family
$$
F:(\R^n)^l \mapsto C^r(\M,\M), \qquad
F
\left( (u_{ij})_{i=1,\ldots,l \atop j=1,\ldots,n}
\right) (x)
= \Phi\left(
f(x), \sum_{i=1}^l\sum_{j=1}^n u_{ij}\cdot X_{ij} ,1
\right)
$$
where $\Phi:T\M\times\R\mapsto\M$ is the geodesic flow
associated to the given Riemannian metric.

We now take a small $\eps_0>0$ and consider the finite
dimensional parameterized family of maps
$F_|:V(0,\eps_0) \mapsto C^r(\M,\M)$, where $V(0,\eps_0)$ is
the $\eps_0$-ball around the origin in $\R^{n\cdot l}$.

Then by construction of $F$, every family
$\F_{a,\eps}=\{ F_t: \| t-a \| <\eps \}$ satisfies
conditions \eqref{eq:random-perturbation} and
\eqref{eq:noise-distribution} for some $\rho_0,\un{q}>0$,
where $\epsilon>0$ is so that
$\omega^*:=a\in\overline{V(a,\eps)}\subset V(0,\eps_0)$ and
we set $\theta=\frac{\l\mid V(a,\eps)}{\l(V (a,\eps))}$.

On paralellizable manifolds the implementation is even
simpler since we can perform the previous construction with
$l=1$ and obtain so called \emph{additive random
  perturbations}, as follows.

\subsubsection{Additive random perturbations on
  parallelizable manifolds}
\label{sec:additive-random-pert}

If $\M$ is parallelizable, then $T\M$ is diffeomorphic to
the trivial bundle $\M \times \RR^n$ and we can find $n$
globally orthonormal (hence nonvanishing) smooth vector
fields $X_1,\dots, X_n$ on $\M$.  We construct the following
family of differentiable maps
$$
F:\R^n \mapsto C^r(\M,\M), \qquad
F(u_1,\dots,u_n)(x)
= \Phi\left(
f(x), \sum_{j=1}^n u_j\cdot X_{j} ,1 \right)
$$
where $\Phi:T\M\times\R\mapsto\M$ is as above. Now for all
small enough $\eps_0>0$ considering
$F_|:V(0,\eps_0) \mapsto C^r(\M,\M)$ where $V(0,\eps_0)$ is
the $\epsilon_0$-ball around the origin in $\R^{n}$, the
family $\F_{a,\eps}=\{ F_t: \| t-a \| <\eps \}$ satisfies
conditions \eqref{eq:random-perturbation} and
\eqref{eq:noise-distribution} for some $\rho_0,\un{q}>0$.

\subsubsection{Representantion of Markov Chains by random
  maps}
\label{sec:repres-markov-chains}

In our setting the random perturbation is a special case of
a Markov Chain with transition probabilities given by
\begin{equation}
\label{eq:transition-prob}
p_x(A)=p(A\mid x)
=\theta\{\omega\in V_\eps(0): f_\omega(x)\in
A\}=[(f^x)_*\theta](A), \quad A\in\B
\end{equation} 
and $\mu$ is a stationary measure for the Markov Chain with
the family $(q_x)_{x\in\M}$ of transition densities
\cite{ohno1983}.

Recently Jost, Kell and Rodrigues in \cite[Theorem B and
Theorem C]{JKR15} showed that Markov Chains with regular
enough transition densities can be represented by continuous
random maps or random diffeomorphisms. 

We state below a result giving sufficient conditions
for representation by continuous random maps.

\begin{theorem}{\cite[Proposition 5.1]{JKR15}}
\label{thm:jost-kell-rodrigues}
Let $(p_x)_{x\in\M}$ be a family of probability measures,
where each $p_x$ is absolutely continuous with respect to
$\l$ and has positive H\"{o}lder continuous (for some
exponent $\alpha >0$) probability density $q_x$ and the
family $(q_x)_{x\in\M}$ is pointwise continuous for
$\l$-a.e. $x\in\M$.  Then $(p_x)_{x\in\M}$ can be
represented by random continuous maps
$(f_{\omega})_{\omega\in\Omega}$, that is, there exists a
probability measure $\nu$ on $C^0(\M,\M)$ so that
$p_x(A)=p(A|x) = \nu\{g : g(x) \in A \}=\int_A q_x\,\dif\l(x)$
for every Borel subset $A$.
\end{theorem}

\subsection{Stationary probability measures. Decay of correlations}
\label{sec:st-meas-dc}

In this setting the notion of \emph{stationary measure}
replaces the notion of invariant measure by leaving the
perturbed map invariant in average over the
noise.
\begin{definition}[Stationary measure]
\label{def:stationary-measure}
We say that the probability measure $\mu$ on the Borel
subsets of $\M$ is stationary if
$\iint \mu(\varphi\circ f_{\omega}) \,\dif\theta(\omega)=\int
\varphi\,\dif\mu$
for every $\mu$-integrable $\varphi:\M\to\R$.
\end{definition}

We can give a deterministic representation of this random
setting using the skew product transformation
\begin{align}
 \label{def:skew-product}
S: \M\times \Omega \to \M\times \Omega, \qquad
(x,\o) \mapsto  (f_{\omega_1},\sigma(\o)),
\end{align}
where $\sigma:\Omega\to\Omega$ is the one-sided shift
$\sigma(\o)=\sigma(\omega_1,\omega_2,\ldots)=(\omega_2,
\omega_3, \ldots)$.
We remark that $\mu$ is stationary if and only if the
product measure $\mu\times \theta^\N$ is an $S$-invariant
measure.

Now the process is given by
\begin{equation}
\label{eq:def-rand-stat-stoch-proc-RDS2} 
X_n=\varphi\circ f^n_{\o},
\quad \mbox{for each } n\in\N,
\end{equation}
which can also be written as $X_n=\varphi\circ\pi\circ S^n$,
where $\pi:\M\times \Omega \to\M, (x,\o) \mapsto x$ is the
natural projection onto the first factor.  Note that the
stochastic process $X_0,X_1,\ldots$ is stationary since
$\mu$ is stationary.


Hence, the random evolution of the system is given by
a discrete time dynamical system $(\X,\B,\p,T)$,
where $\X$ is a topological space, $\mathcal B$ is the Borel
$\sigma$-algebra, $T:\X\to\X$ is a measurable map and $\p$
is a $T$-invariant probability measure, \ie
$\p(T^{-1}(B))=\p(B)$, for all $B\in \mathcal B$. 
We set $(\X,\mathcal B, \p, T)$ with $\X=\M\times \Omega$
and the product Borel $\sigma$-algebra $\mathcal B$ where
the product measure $\p=\mu\times\theta^\N$ is defined. The
random dynamics can now be read from the skew product map
$T=S$ since the second factor of $S$ depends only on the
first coordinate of the first factor
\begin{align*}
p_xA=\p\{(x,\o):(\pi\circ S)(x,\o)\in A\}
=[(\pi\circ S)_*\p]A=[(f^x)^*\theta]A.
\end{align*}




In our (random) setting, we will only be interested in
Banach spaces of functions that do not depend on
$\o\in\Omega$. Hence, we assume that $\phi,\psi$ are
functions defined on $\M$ and the correlation
between these two observables can be written in a simple
form.

\begin{definition}[Annealed decay of correlations]
\label{def:ann-cor}
Let \( \mathcal C_{1}, \mathcal C_{2} \)
denote Banach spaces of real valued measurable functions
defined on \( \M \).
We denote the \emph{annealed correlation} of non-zero
functions $\phi\in \mathcal C_{1}$ and
\( \psi\in \mathcal C_{2} \)
w.r.t.\ the measure $\mu\times\theta^\N$ as
\begin{equation} \label{eq:cor-random}
  \cv_{\mu\times\theta^\N}(\phi,\psi, n)\coloneqq\frac{1}{\|\phi\|_{\mathcal
      C_{1}}\|\psi\|_{\mathcal C_{2}}} \left|\int \left(\int \psi\circ
      f^n_{\o}\, \dif\theta^{\N}\right)\phi\, \dif\mu-\int  \phi\,
    \dif\mu\int \psi\, \dif\mu\right|.
\end{equation}


We say that we have \emph{annealed decay of correlations},
w.r.t.\ the measure $\mu\times\theta^\N$, for observables in
$\mathcal C_1$ \emph{against} observables in $\mathcal C_2$
if, for every $\phi\in\mathcal C_1$ and every
$\psi\in\mathcal C_2$, then it holds that
$\cv_{\mu\times\theta^\N}(\phi,\psi,n)\xrightarrow[n\to\infty]{}
0$.
\end{definition}

We say that we have \emph{annealed decay of correlations
  against $L^1$ observables} whenever we have decay of
correlations, with respect to the measure
$\mu\times\theta^\N$, for observables in $\mathcal C_1$
against observables in $\mathcal C_2$ and
$\mathcal C_2=L^1(\l)$ is the space of $\l$-integrable
functions on $\M$ and
$\|\psi\|_{\mathcal C_{2}}=\|\psi\|_1=\int
|\psi|\,\dif\l$.
Note that when $\mu$ is absolutely continuous
with respect to $\l$ and the respective Radon-Nikodym
derivative is bounded above and below by positive constants,
then
$L^1(\l)=L^1(\mu)$. 



\subsection{Regularity conditions on the observable function
  and the measure}
\label{sec:regular-condit-obser}

We assume that the r.v.\ $\varphi:\M\to\R\cup\{\pm\infty\}$
achieves a global maximum at $\zeta\in \M$ (we allow
$\varphi(\zeta)=+\infty$). 
We also assume that $\varphi$ and $\p$ are sufficiently regular so that:

\begin{enumerate}

\item[\namedlabel{item:U-ball}{(R1)}] 
for $u$ sufficiently close to $u_F \coloneqq\varphi(\zeta)$,  the event 
\begin{equation*}
U(u)=\{X_0>u\}=\{x\in\M:\; \varphi(x)>u\}
\end{equation*} corresponds to a topological ball centred at
$\zeta$. Moreover, the quantity $\p(U(u))$, as a function of
$u$, varies continuously on a neighbourhood of
$u_F$. 

\end{enumerate}

In what follows, an \emph{exceedance} of the level $u\in\R$
at time $j\in\N$ means that the event $\{X_j>u\}$ occurs. We
denote by $F$ the distribution function (d.f.)~of $X_0$, \ie
$F(x)=\p(X_0\leq x)$. Given any d.f.\ $F$, let
$\bar{F}=1-F$, the so-called \emph{tail distribution}, and
$u_F$ denote the right endpoint of the d.f.\ $F$, \ie
$ u_F=\sup\{x: F(x)<1\}.  $

\subsection{Extreme Value Laws}
\label{sec:extreme-value-laws}

Given the dynamically defined
time series $X_0,X_1,\ldots$ we want to study its extremal
behaviour. Hence we define a new sequence of random
variables $M_1, M_2,\ldots$ as the partial maximum of the
first $n$ random variables, \ie
\begin{equation}
\label{eq:Mn-definition}
M_n=\max\{X_0,\ldots,X_{n-1}\}.
\end{equation}

\begin{definition}
  We say that we have an \emph{Extreme Value Law}
  (\emph{EVL}) for $M_n$ if there is a non-degenerate d.f.\ $H:\R\to[0,1]$ with $H(0)=0$; 
  and if for every $\tau>0$, there exists a sequence of levels $u_n=u_n(\tau)$, $n=1,2,\ldots$,  such that
\begin{equation}
\label{eq:un}
  n\,\p(X_0>u_n)\to \tau,\;\mbox{ as $n\to\infty$}
\end{equation}
and for which the following holds
\begin{equation}
\label{eq:EVL-law}
\p(M_n\leq u_n)\to \bar H(\tau),\;\mbox{ as $n\to\infty$.}
\end{equation}
\end{definition}

\vspace{.5cm}
For every sequence $(u_n)_{n\in\N}$ satisfying \eqref{eq:un} we define:
\begin{equation}
\label{def:Un}
U_n \coloneqq\{X_0>u_n\}.
\end{equation} 

The normalising sequences $u_n$ come from the i.i.d.\
case. Namely, if $X_0,X_1,X_2,\ldots$ are independent and
identically distributed, then it is clear that
$\p(M_n\leq u)= (F(u))^n$ where $F$ is the d.f. of
$X_0$. Hence, condition \eqref{eq:un} implies that
\begin{equation}
\label{eq:iid-maxima}
\p(M_n\leq u_n)= (1-\p(X_0>u_n))^n\sim\left(1-\frac\tau
  n\right)^n\to\e^{-\tau}, \quad\text{as $n\to\infty$.}
\end{equation}
Moreover, the reciprocal is also true; see
\cite[Theorem~1.5.1]{LLR83} for more details. Note that in
this case $H(\tau)=1-\e^{-\tau}$ is the standard exponential
d.f.

When $X_0,X_1,X_2,\ldots$ are not independent, the standard
exponential law still applies under some conditions on the
dependence structure. These conditions are as follows.
\begin{condition}[$D_2(u_n)$]\label{cond:D2} We say that $D_2(u_n)$ holds for the sequence $X_0,X_1,\ldots$ if for all $\ell,t$
and $n$
\begin{align*}
|\p\left(X_0>u_n\cap
  \max\{X_{t},\ldots,X_{t+\ell-1}\}\leq u_n\right)-\p(X_0>u_n)
  \p(M_{\ell}\leq u_n)|\leq \gamma(n,t),
\end{align*}
where $\gamma(n,t)$ is decreasing in $t$ for each $n$ and
$n\gamma(n,t_n)\to0$ when $n\rightarrow\infty$ for some
sequence $t_n=o(n)$.
\end{condition}
Now, let $(k_n)_{n\in\N}$ be a sequence of integers such that
\begin{equation}
\label{eq:kn-sequence-1}
k_n\to\infty\quad \mbox{and}\quad  k_n t_n = o(n).
\end{equation}
\begin{condition}[$D'(u_n)$]\label{cond:D'} We say that $D'(u_n)$
holds for the sequence $X_0, X_1, X_2, \ldots$ if there exists a sequence $(k_n)_{n\in\N}$ satisfying \eqref{eq:kn-sequence-1} and such that
\begin{equation}
\label{eq:D'un}
\lim_{n\rightarrow\infty}\,n\sum_{j=1}^{\lfloor n/k_n \rfloor}\p( X_0>u_n,X_j>u_n)=0.
\end{equation}
\end{condition}
By \cite[Theorem~1]{FF08a}, if conditions $D_2(u_n)$ and
$D'(u_n)$ hold for $X_0, X_1,\ldots$ then there exists an
EVL for $M_n$ and $H(\tau)=1-\e^{-\tau}$. Besides, as it can
be seen in \cite[Section~2]{FF08a}, condition $D_2(u_n)$
follows immediately if $X_0, X_1,\ldots$ is given by
\eqref{eq:def-stat-stoch-proc-DS} and the system has
sufficiently fast decay of correlations.

In this paper, we extend the results for the random case
from Aytaç, Freitas and Vaienti \cite{AFV15} to transitive
maps and for more general random perturbations.  The class
of maps that can be perturbed includes interval maps with
unbounded derivatives, as in the Lorenz map for finite
branch case; interval maps with infinitely many branches,
and also maps in higher dimensional manifolds. The random
perturbation setting is not restricted to additive noise as
in \cite{AFV15}. Moreover, we also consider Markov Chains
under conditions that guarantee their representation by
random maps satisfying conditions
\eqref{eq:random-perturbation} and
\eqref{eq:noise-distribution}, which imply an almost uniform
distribution of perturbed images on a neighbourhood of the
original value of the unperturbed map.

First of all, we show that, for transitive systems, we have
decay of correlations against $L^1$ observables for the type
of random perturbations we consider here.

\begin{maintheorem}\label{mthm:trans->DC}
  Every measurable map $f:\M\to\M$ which is
  $\l$-a.e.\ continuous admitting $x_0\in\M$ so that
  $\{f^nx_0:n\ge1\}$ is both a dense subset of $\M$ and a set of
  continuity points of $f$ is such that any random
  perturbation of $f$ satisfying
  \eqref{eq:random-perturbation} and
  (\ref{eq:noise-distribution}) has exponential decay of
  correlations against $L^1$ observables.
\end{maintheorem}

We remark that the assumptions on the underlying unperturbed
dynamics are very weak and the conclusion in
Theorem~\ref{mthm:trans->DC} is rather strong.

The assumptions on the random perturbation in
Theorem~\ref{thm:jost-kell-rodrigues} ensure that
$\supp (p_x)\supset B(fx,\rho(x))$, that is, the perturbed
images cover a full neighbourhood of the image of the
original map $f$; and $q_x\mid B(fx,\rho(x))\ge g(x)$, i.e.,
the distribution of images in this neighbourhood is
essentially uniform, for some $\l$-a.e.\ continuous map
$f:\M\to\M$ and continuous functions $\rho,g:\M\to\R^+$.
Hence we can state the following version of
Theorem~\ref{mthm:trans->DC} in the language of Markov
chains.

\begin{maintheorem}
  \label{mthm:random-pert-decay}
  Let $(p_x)_{x\in\M}$ be a continuous family of probability
  measures such that $p_x=q_x\l$, where $q_x$ is a positive
  H\"{o}lder continuous (for some exponent $\alpha >0$)
  probability density $q_x$ varying continuously with
  $x\in\M$ with respect to the $C^0$-topology. Assume that
  there are $\rho_0>0,\ov{q}>\un{q}>0$ and a full $\l$-measure
  subset $Y$ so that the map $f:\M\to\M$ is continuous on
  $Y$ satisfying $\supp(q_x)\supset B(fx,\rho_0)$ and
  $\un{q}\le q_x\le\ov{q}$ for $x\in Y$ and admitting a
  point $x_0\in\M$ so that $\{f^nx_0: n\ge0\}$ is both dense
  in $\M$ and contained in $Y$.

  Then the Markov Chain defined by $(p_x)_{x\in\M}$ has a
  unique stationary measure $\mu$ with exponential decay of
  correlations against $L^1$ observables.
\end{maintheorem}

Using this general result on decay of correlations for
random maps/Markov Chains we show that under suitable random
perturbation of the original transitive system, we get a
standard exponential distribution for the extreme values as
well as the hitting time statistics for any point
$\zeta\in\M$.

\begin{maintheorem}
\label{mthm:random-EVL-finite}
Let $(\M\times \Omega, \mathcal B, \mu\times\theta^\N, S)$
be a dynamical system where $\M$ is a finite dimensional
compact Riemannian manifold and
$f:\M\to \M$ is a map which is continuous on the full
$\l$-measure subset $Y$ admitting a point $x_0\in\M$ so that
$\{f^nx_0: n\ge0\}$ is both dense in $\M$ and a contained in
$Y$.  Assume that $f$ is randomly perturbed by a random maps
scheme satisfying \eqref{eq:random-perturbation} and
\eqref{eq:noise-distribution} or by a Markov Chain given by
a family $(p_x)_{x\in\M}$ of transition probabilities
satisfying the conditions of
Theorem~\ref{mthm:random-pert-decay}.

For any point $\zeta\in\M$, consider that
$X_0, X_1,\ldots$ is defined as in
\eqref{eq:def-rand-stat-stoch-proc-RDS2}, let $u_n$ be such
that \eqref{eq:un} holds and assume that $U_n$ is defined as
in \eqref{def:Un}.

Then the stochastic process $X_0, X_1,\ldots$ satisfies
$D_2(u_n)$ and $D'(u_n)$, which implies that we have an EVL
for $M_n$ such that $\bar H(\tau)=\e^{-\tau}$.
\end{maintheorem}

\subsection{Hitting/Return Time Statistics}
\label{sec:hittingr-time-statis}

Next we consider the second approach in the statistical
study of rare events. In the deterministic setting the
definition of first hitting/return time (function) is given
as follows.

Given a set $A\in\B$ we define a function that we refer to
as \emph{first hitting time function} to $A$ and denote by
$r_A:\X\to\N\cup\{+\infty\}$ where
\begin{equation*}
r_A(x)=\min\left\{j\in\N\cup\{+\infty\}:\; f^j(x)\in A\right\}.
\end{equation*}
The restriction of $r_A$ to $A$ is called the \emph{first
  return time function} to $A$. We define the \emph{first
  return time} to $A$, which we denote by $R(A)$, as the
minimum of the return time function to $A$, \ie
\[
R(A)=\min_{x\in A} r_A(x).
\]
In the random setting, one has to a make choice regarding
the type of definition for the first hitting/return times
(functions). Essentially, there are two possibilities. The
\emph{quenched} perspective which consists of fixing a
realisation $\o\in\Omega$ and defining the objects in the
same way as in the deterministic case. The \emph{annealed}
perspective consists of defining the same objects by
averaging over all possible realisations $\o$. (We refer to
\cite{MR11} for more details on both perspectives.)  In
\cite{AFV15}, the quenched perspective was used to define
hitting/return times (functions) as it facilitates the
connection between EVL and Hitting/Return Time Statistics in
the random setting. Here, we follow the same setting.

For some $\o\in\Omega$ fixed, some $x\in \M$ and
$A\subset \M$ measurable, we define the \emph{first random
  hitting time}
\begin{equation*}
\label{eq:rht}
r_A^{\o}(x) \coloneqq \min\{j\in\N:\; f_{\o}^j(x)\in A\}
\end{equation*}
and the \emph{first random return} from A to A as
\begin{equation*}
\label{eq:rrt}
R^{\o}(A)=\min\{r_A^{\o}(x):\; x\in A\}.
\end{equation*}
\begin{definition}
\label{def:HTS/RTS}
Given a sequence of measurable subsets of $\X$,
$(V_n)_{n\in \N}$, such as $\p(V_n)\to 0$, the system has
(random) \emph{Hitting Time Statistics} (\emph{HTS}) $G$ for
$(V_n)_{n\in \N}$ if for all $t\ge 0$
\begin{equation}\label{eq:def-HTS-law}
 \p\left(r_{V_n}\leq\frac t{\p(V_n)}\right)\to G(t) \;\mbox{ as $n\to\infty$,}
\end{equation}
and the system has (random) \emph{Return Time Statistics}
(\emph{RTS}) $\tilde G$ for $(V_n)_{n\in \N}$ if for all
$t\ge 0$
\begin{equation}\label{eq:def-RTS-law}
  \p_{V_n}\left(r_{V_n}\leq\frac t{\p(V_n)}\right)\to\tilde G (t)\;\mbox{ as $n\to\infty$}.
\end{equation}
We observe that in the random setting,
$\X=\M\times\Omega$, $\p=\mu\times\theta^\N$, $T=S$ as defined
in \eqref{def:skew-product}, $V_n=V_n^*\times\Omega$,
where $V_n^*\subset \M$ and $\mu(V_n^*)\to 0$ as
$n\to \infty$.
\end{definition}
We also observe that
\[
\p\left(r_{V_n}\leq\frac{t}{\p(V_n)}\right)
=
\mu\times\theta^\N
\left(r^{\o}_{V_n^*}\leq\frac{t}{\mu(V_n^*)}\right).
\]
The normalising sequences to obtain HTS/RTS, are motivated
by Kac's Lemma. It asserts that the expected value of $r_A$
w.r.t.\ $\p_A$ is $\int_A r_A~d\p_A =1/\p(A)$. So, the
appropriate normalising factor in the study of the
fluctuations of $r_A$ on $A$ is $1/\p(A)$.

The relation between the existence of HTS and that of RTS is
given by the Main Theorem in \cite{HLV05}. It states that a
system has HTS $G$ if and only if it has RTS $\tilde G$ and
\begin{equation}
\label{eq:HTS-RTS}
G(t)=\int_0^t(1-\tilde G(s))\,\dif s.
\end{equation}
So, the existence of exponential HTS is equivalent to the
existence of exponential
RTS. 

In \cite{FFT10}, the link between HTS/RTS (for balls) and
EVLs of stochastic processes given by
\eqref{eq:def-stat-stoch-proc-DS} was established for
invariant measures $\upsilon$ absolutely continuous w.r.t.\
$\l$. Essentially, it was proved that if such time series
have an EVL $H$ then the system has HTS $H$ for balls
``centred'' at $\zeta$ and vice
versa. 
(Recall that having HTS $H$ is equivalent to saying that the
system has RTS $\tilde H$, where $H$ and $\tilde H$ are
related by \eqref{eq:HTS-RTS}). This was based on the
elementary observation that for stochastic processes given
by \eqref{eq:def-stat-stoch-proc-DS} we have:
\begin{equation}
\label{eq:rel-Mn-r}
f^{-1}(\{M_n\leq u\})=\{r_{\{X_0>u\}}>n\}.
\end{equation}
This connection was exploited to prove EVLs using tools from
HTS/RTS and the other way around. In \cite{FFT11}, the
authors carried the connection further to include more
general measures, which, in particular, allowed the
coauthors to obtain the connection in the random setting in
\cite{AFV15}. For that, it was sufficient to use the skew
product map to look at the random setting as a deterministic
system and to take the observable
$\varphi\circ\pi:\M\times\Omega\to\R\cup\{+\infty\}$ defined
as in (\ref{eq:def-rand-stat-stoch-proc-RDS2}) with
$\varphi:\M\to\R\cup\{+\infty\}$ as in \cite[equation
(4.1)]{FFT11}. Namely,
\begin{equation} \label{eq:observation} 
\varphi:\M\to \R\cup\{+\infty\}, \qquad
x\mapsto g\left(\p(B_{\dist(x,\zeta)}(\zeta))\right)
\end{equation}
where $\zeta$ is a chosen point in the phase space $\M$ and
the function $g:[0,+\infty)\to\R\cup\{+\infty\}$ is such
that $0$ is a global maximum ($+\infty$ is allowed), $g$ is
a strictly decreasing bijection in a neighbourhood of $0$
and has one of the three types coming from the Classical
Extreme Value Theory; see \cite{FFT11}.
Then \cite[Theorems~1 and 2]{FFT11} guarantee that if we
have an EVL, in the sense that \eqref{eq:EVL-law} holds for
some d.f.\ $H$, then we have HTS for sequences
$\{V_n\}_{n\in\N}$, where $V_n=B_{\delta_n}\times \Omega$
and $\delta_n\to0$ as $n\to\infty$, with $G=H$ and
viceversa.

Using the connection between EVLs and HTS/RTS provided by
\cite{AFV15}, we immediately get from
Theorem~\ref{mthm:random-EVL-finite}

 \begin{maincorollary}
   \label{cor:random-EVL=>HTS} Under the same hypothesis of
   Theorem~\ref{mthm:random-EVL-finite} we have exponential
   HTS/RTS for balls around $\zeta$, in the sense that
   \eqref{eq:def-HTS-law} and \eqref{eq:def-RTS-law} hold
   with $G(t)=\tilde G(t)=1-\e^{-\t}$ and
   $V_n=B_{\delta_n}(\zeta)\times\Omega$, where
   $\delta_n\to0$, as $n\to\infty$.
 \end{maincorollary}

%

\subsection{Rare Event Point Processes}
\label{sec:rare-event-point}

If we consider multiple exceedances we are lead to point
processes of rare events counting the number of exceedances
in a certain time frame. For every $A\subset\R$ we define
\[
\nn_u(A):=\sum_{i\in A\cap\N_0}\I_{\{X_i>u\}}.
\]
In the particular case where $A=I=[a,b)$ we simply write 
$\nn_{u,a}^b:=\nn_u([a,b)).$ 
Observe that $\nn_{u,0}^n$ counts the number of exceedances amongst the first $n$ observations of the process $X_0,X_1,\ldots,X_n$ or, in other words, the number of entrances in $U(u)$ up to time $n$. Also, note that
\begin{equation}
\label{eq:rel-HTS-EVL-pp}
\{\nn_{u,0}^n=0\}=\{M_n\leq u\}.
\end{equation}
In order to define a point process that captures the essence
of an EVL and HTS through \eqref{eq:rel-HTS-EVL-pp}, we need
to re-scale time using the factor $v:=1/\p(X>u)$ given by
Kac's Theorem. However, before we give the definition, we
need some formalism. 

Let $\G$ denote the semi-ring of subsets of $\R_0^+$ whose
elements are intervals of the type $[a,b)$, for
$a,b\in\R_0^+$. Let $\RR$ denote the ring generated by
$\G$. Recall that for every $J\in\RR$ there are $k\in\N$ and
$k$ intervals $I_1,\ldots,I_k\in\G$ such that
$J=\cup_{i=1}^k I_j$. In order to fix notation, let
$a_j,b_j\in\R_0^+$ be such that $I_j=[a_j,b_j)\in\G$. For
$I=[a,b)\in\G$ and $\alpha\in \R$, we denote
$\alpha I:=[\alpha a,\alpha b)$ and
$I+\alpha:=[a+\alpha,b+\alpha)$. Similarly, for $J\in\RR$
define $\alpha J:=\alpha I_1\cup\cdots\cup \alpha I_k$ and
$J+\alpha:=(I_1+\alpha)\cup\cdots\cup (I_k+\alpha)$.

\begin{definition}
  We define the \emph{rare event point process} (REPP) by
  counting the number of exceedances (or hits to $U(u_n)$)
  during the (re-scaled) time period $v_nJ\in\RR$, where
  $J\in\RR$. To be more precise, for every $J\in\RR$, set
\begin{equation}
\label{eq:def-REPP}
N_n(J):=\nn_{u_n}(v_nJ)=\sum_{j\in v_nJ\cap\N_0}\I_{\{X_j>u_n\}}.
\end{equation}
\end{definition}

When $D'(u_n)$ holds then, since there is no clustering, due
to a criterion proposed by Kallenberg
\cite[Theorem~4.7]{Ki86} which applies only to simple point
processes (without multiple events), we can adjust condition
$D_2(u_n)$ to this scenario of multiple exceedances in order
to prove that the REPP converges in distribution to a
standard Poisson process. We denote this adapted condition
by:
\begin{condition}[$D_3(u_n)$]\label{cond:D^*} Let $A\in\RR$ and $t\in\N$.
  We say that $D_3(u_n)$ holds for the sequence
  $X_0,X_1,\ldots$ if
\[ \left|\p\left(\{X_0>u_n\}\cap
  \{\nn(A+t)=0\}\right)-\p(\{X_0>u_n\})
  \p(\nn(A)=0)\right|\leq \gamma(n,t),
\]
where $\gamma(n,t)$ is nonincreasing in $t$ for each $n$ and
$n\gamma(n,t_n)\to0$ as $n\rightarrow\infty$ for some sequence
$t_n=o(n)$, which means that $t_n/n\to0$ as $n\to \infty$.
\end{condition}
Condition $D_3(u_n)$ follows, as easily as $D_2(u_n)$, from
sufficiently fast decay of correlations.

In \cite[Theorem~5]{FFT10} a strengthening of
\cite[Theorem~1]{FF08a} is proved, which essentially says
that, under $D_3(u_n)$ and $D'(u_n)$, the REPP $N_n$ defined
in \eqref{eq:def-REPP} converges in distribution to a
standard Poisson process.

Since, under the same assumptions of
Theorem~\ref{mthm:random-EVL-finite}, condition $D_3(u_n)$
holds trivially, then applying \cite[Theorem~5]{FFT10} we
obtain

 \begin{maincorollary}
   \label{cor:random-EVL=>Poisson} 
   Under the same hypothesis of
   Theorem~\ref{mthm:random-EVL-finite}, the stochastic
   process $X_0, X_1,\ldots$ satisfies $D_3(u_n)$ and
   $D'(u_n)$, which implies that the REPP $N_n$ defined in
   \eqref{eq:def-REPP} is such that $N_n\xrightarrow[]{d}N$,
   as $n\rightarrow\infty$, where $N$ denotes a Poisson
   Process with intensity $1$.
 \end{maincorollary}

\subsection{Organization of the work}
\label{sec:organization-work}

We present some examples of applications focusing on
specific families for which our results directly improve
recent advances, in Section~\ref{sec:applications}.

In Section~\ref{sec:Harris->DC} we explain how our setting
fits into a Markov Chain satisfying Harris and Doeblin
conditions, enabling us to obtain exponential decay of correlations from
already known results.
Then, in Section~\ref{sec:proof-main-results}, we prove
the main results on extreme value laws, hitting/return time statistics
and rare event point processes.

\subsection*{Acknowledgements}
HA would like to thank Tertuliano Franco (UFBA) for very
useful discussions concerning Harris chains. HA was
supported partly by CNPq (Brazil) grant number
162724/2013-6, and by Programa Nacional de P\'os-Doutorado
(PNPD), CAPES (Brazil). VA was partially supported by CAPES,
CNPq (Project 301392/2015-3) and FAPESB (Brazil).

\section{Applications}
\label{sec:applications}

Here we give some examples for which we can apply our
results. The assumptions are rather weak: any transitive map
(that is, admitting a subset of points with dense forward
orbit) of a compact manifold can be randomly perturbed in
our setting.

In what follows we focus on specific families for which our
results directly improve recent advances.

\subsection{Lorenz-like maps}\label{ex:ex-1-1}

Lorenz maps are the one-dimensional maps associated to the geometric
Lorenz models, which were constructed as an
attempt to understand the numerically observed behaviour of
the Lorenz attractor introduced by Lorenz in
\cite{Lo63}. The Lorenz equations
\begin{equation}\label{eq:Lorenz}
\dot x=a(y-x),\, \dot y=(r-z)x-y, \, \dot z=xy-bz,
\end{equation}
with the parameters $a=10,\, r=28/3 \textrm { and } b=8/3$
were intended as an extremely simplified model for thermal
fluid convection, in order to understand the atmospherical
circulation. Numerical simulations for an open neighbourhood
of these values of the parameters pointed to the existence
of a strange attractor, but this non-linear system of
differential equations poses both numerical and analytical
challenges to its understanding. Ten years after the
introduction of this system, the so-called \textit{geometric
  Lorenz models} were constructed as an attempt to
rigorously understand the phenomena observed by Lorenz. They
were proposed by Afraimovich, Bykov, Shil'nikov in
\cite{ABS77} and Guckenheimer, Williams in \cite{GW79},
independently. These models are three-dimensional flows for
which it is possible to prove the existence of a strange
attractor with regular solutions accumulating a singular (or
an equilibrium) point. Moreover, this attractor is sensitive
to initial conditions and can not be destroyed by small
perturbations of the original flow, that is to say it is
robust. Finally, Tucker \cite{Tu99,Tu02} proved the existence and
robustness of the Lorenz attractor and, as a consequence of
the method of his proof, showed that these models do
describe the behaviour of \eqref{eq:Lorenz}. For more
information on the history of the subject and the
construction of the geometric models, we refer the reader to
Araujo, Pacifico and Viana \cite{AraPac2010,Vi00} and references
therein.

Basically, the study of the geometric Lorenz flows is done
through the reduction to a Poincar\'e first return map to a
global singular two-dimensional cross-section, which is then
further reduced to the study of a one-dimensional
transformation. This one-dimensional transformation is
obtained by quotienting the return map over an invariant
contracting foliation by curves which partition the
cross-section. The map $f$ satisfies (see
Figure~\ref{fig:Lorenz1D})
\begin{enumerate}
\item $f$ is discontinuous at $x = 0$ with $\lim_{x\to 0^{-}} f(x) = +1$ and $\lim_{x\to 0^{+}} f(x) = -1$;
\item $f$ is differentiable on $[-1/2,0)\cup (0,1/2]$,
  $f'(x) > \sqrt{2}$ and
  $\lim_{x\to0^{\pm}}\frac{f'(x)}{x^\beta}=\alpha_\pm$,
  where $0<\beta<1$;
\item $f$ is topologically exact and thus transitive.
\end{enumerate}

\begin{figure}[tb]
 \begin{center}
    \includegraphics[width=14cm]{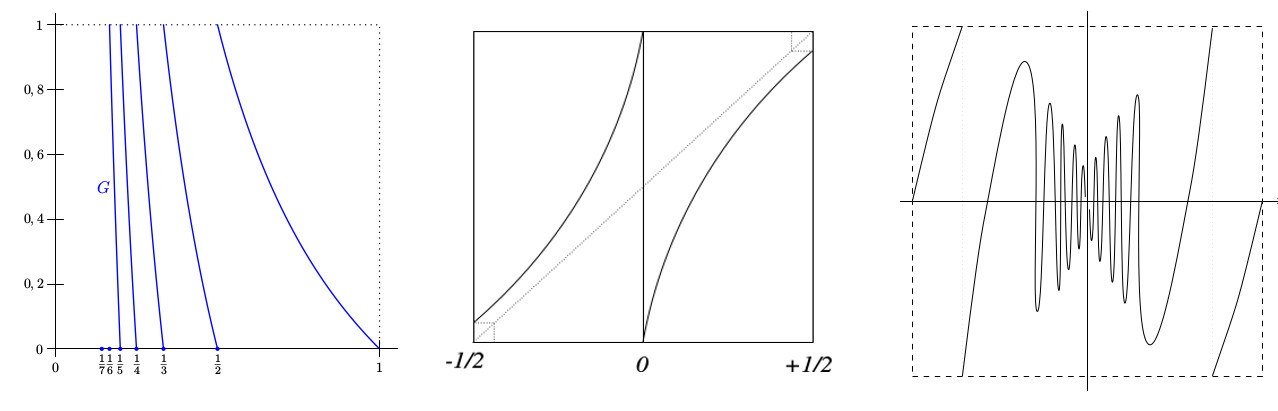}
 \end{center}
 \caption{\label{fig:Lorenz1D}
From left to right: Gauss map, Lorenz map, infinite modal map.}
\end{figure}

In \cite{GHN11}, Gupta et al.\ established exponential
limiting laws for the extremal study of Lorenz-like maps in
the deterministic setting. Later, in \cite{FFLTV13}, Faranda
et al.\ gave numerical results for additive random
perturbations of a family of Lorenz maps, pointing to the
convergence of extreme values to the classical EVL
distributions for increasing values of the
noise. 
Here, we give an analytic solution for an arbitrary noise
level as a result of
Theorems~\ref{mthm:trans->DC},~\ref{mthm:random-EVL-finite}
and Corollaries~\ref{cor:random-EVL=>HTS},
~\ref{cor:random-EVL=>Poisson}.


\begin{maincorollary}
\label{cor:Lorenz-map}
Let $f:{\mathcal S}^1\to {\mathcal S}^1$ be a map satisfying
conditions (1)-(3) listed above, which is randomly perturbed
as in \eqref{eq:random-perturbation} with noise distribution
given by \eqref{eq:noise-distribution}. For any point
$\zeta\in\M$, consider that $X_0, X_1,\ldots$ is defined as
in \eqref{eq:def-rand-stat-stoch-proc-RDS2} and let $u_n$ be
such that \eqref{eq:un} holds. Then the stochastic process
$X_0, X_1,\ldots$ satisfies $D_2(u_n)$, $D_3(u_n)$ and
$D'(u_n)$, which implies that we have an EVL for $M_n$ such
that $\bar H(\tau)=\e^{-\tau}$ and we have exponential
HTS/RTS for balls around $\zeta$. Moreover, the REPP $N_n$
defined in \eqref{eq:def-REPP} is such that
$N_n\xrightarrow[]{d}N$, as $n\rightarrow\infty$, where $N$
denotes a Poisson Process with intensity $1$.
\end{maincorollary}

\subsection{Countable branch case} 
\label{sec:countable-branch-cas}

We can also apply our results to full branch Markov maps
with countable number of branches, like the Gauss map or
maps in the setting of Rychlik's Theorem \cite{Ry83}, as
studied in \cite[Section 3.2.1]{AFV15} in the deterministic
setting, since these classes of maps are transitive. Thus
our results about EVL, HTS/RTS and REPP also hold for this
class of systems.

\subsection{Smooth interval maps}
\label{sec:benedicks-carles-map}

In \cite{BC85,BC91} Benedicks and Carleson proved the
existence of a positive Lebesgue measure subset of
parameters $P\subset[1,2]$ of the family of quadratic maps
$x\mapsto f_a(x)=a-x^2$ for which there exists an absolutely
continuous $f_a$-invariant probability measure (acim)
$\mu_a$ and $f_a$ is topologically mixing on the support
$I_a=[f_a(0),f_a^2(0)]$ of $\mu_a$ for $a\in P$. In
particular, these maps admit a dense forward orbit on the
interval $I_a$.

In \cite{FF08}, Freitas and Freitas  studied the
extremal behaviour of Benedicks-Carleson (BC) maps in
deterministic case and they got a standard exponential
extreme value law. Later, in \cite[Section 4.3]{FFLTV13},
the authors numerically studied the additive random
perturbations of quadratic maps, which includes BC maps, and
concluded that under suitable normalisations one should get
standard exponential laws as well.

In fact, Lyubich~\cite{Ly2002} shows that almost all
quadratic maps either admit an ergodic absolutely continuous
invariant probability measure, where we can apply our
results, or there exists a periodic sink whose basin covers
the interval except a zero Lebesgue measure subset. This is
typical of unimodal families \cite{AvMor05}. In particular
no requirements on decay of correlations are needed and so
our results can also be applied to multimodal maps
exhibiting an absolutely continuous invariant probability
measure.

Applying our results, we can get the standard exponential
limiting laws under random perturbations
analytically. Moreover, we get results for HTS/RTS and REPP,
as expected from the previous numerical studies. 

\subsection{Infinite modal maps}
\label{sec:infinite-modal-maps}

In \cite{PRV98,ArPa04} certain parametrized families of
one-dimensional maps with infinitely many critical points
were analyzed from the measure-theoretical point of view;
see Figure~\ref{fig:Lorenz1D}. 

It was proved that such families admit absolutely continuous
invariant probability measures and are topologically mixing
(in particular topologically transitive and so admit a dense
orbit) for a positive Lebesgue measure subset of
parameters. Moreover, the densities of these measures vary
continuously with the parameter and each measure exhibits
exponential rate of mixing for H\"older observables.

We can apply our results to each map in these families
obtaining EVL, HTS/RTS and REPP in the setting of the main
results stated.

\subsection{Piecewise expanding maps in higher dimensions}
\label{sec:piecew-expand-maps}

Again, we can apply all our main results on EVL, HTS/RTS and
REPP to piecewise expanding maps in higher dimensions in the
setting of \cite{Sa00}, as studied in \cite{AFV15} in the
deterministic setting.


\section{Markov chains and decay against $L^1$ for random
  perturbations}
\label{sec:Harris->DC}

In this section we show that random perturbations, as defined
in Section~\ref{sec:statement-results}, of a transitive
dynamical system define a Harris chain which also satisfies
Doeblin's Condition, from which we get fast decay of
correlations against all $L^1$ observables as a consequence
of exponentially fast convergence to the equilibrium or
stationary distribution.

We have already seen in
Subsection~\ref{sec:repres-markov-chains} that our random
perturbations are a particular case of a Markov Chain with
transition probabilities given by
\eqref{eq:transition-prob}, with transition densities
$(q_x)_{x\in\M}$ and stationary measure $\P=\mu$.

\subsection{Harris and Doeblin conditions}
\label{sec:harris-doeblin-condi}

Roughly speaking, a Harris chain is a Markov chain that
returns to a particular part of the state space an unbounded
number of times with positive probability. For the precise
definition we follow Durrett \cite{Du10}.

We denote the $n$th-step transition probability by
\begin{align*}
  p_x^nA=[(f^x)^n_*\theta^n](A)=\int 1_A\big(f^n_\omega
  x)\,\dif\theta^n(\omega)= \int_Aq_x^n\,\dif\l, \quad n\ge1.
\end{align*}

\begin{definition}[Harris chain]
\label{def:Harris-chain}
A Markov chain $\Phi_n$ is a \textit{Harris chain} if one
can find measurable sets $A,B\in\B$, a function $g$ and a
constant $\xi>0$ with $g(x,y)\geq\xi$ for $x\in A,\, y\in B$
and a probability measure $m$ concentrated on $B$ such that
\begin{itemize}
\item[i)] $p_z(\{x: \tau_A(x)<\infty\})>0$ for
  $\l$-a.e. $z$, where
  $\tau_A(x)= \inf\{n\geq 0: p^n_xA>0\}$;
\item[ii)] $x\in A$,
  $C\subset B \implies p_xC\geq \int_C g(x,y)\,\dif m(y)$.
\end{itemize}
\end{definition}



A probability measure $\pi$ on $\M$ with
$\int p_xA\cdot\pi(dx)=\pi(A)$, for all $A\in\B$, is a
\emph{stationary distribution} for a Markov chain. As proved
in \cite{H56}, there exists a unique stationary distribution
for Harris chains. In our setting
$\pi=\p=\mu\times\theta^\N$.

\begin{definition}[Aperiodicity]
\label{def:aperiodicity}
A Markov chain is \emph{aperiodic} if there is no partition
$\M=\sqcup_{i=1}^{\ell}\M_i$ (where $\sqcup$ represents a
disjoint union) for some $\ell\geq 2$ which satisfies
\begin{align*}
  p_x\M_{i+1}=1, \quad\forall x\in\M_i,
  i=1,\ldots,\ell-1
  \qand
  p_x\M_1=1, \quad\forall x\in\M_\ell.
\end{align*}
\end{definition}

Aperiodicity ensures \emph{ergodicity} of the Markov Chain:
the state space admits no decomposition into strictly
smaller sets which are invariant under finitely many
iterates of the process.

\begin{definition}[Doeblin's Condition]
\label{def: Doeblin's condition}
There are $0<\gamma<1$, $\delta>0$, an integer $k\ge1$ and a
probability measure $m$ so that
$ m(A)>\gamma\implies p^k_xA\geq\delta $ for any measurable
set $A$ and $m$-a.e. $x$.
\end{definition}

This condition together with the previous one ensures
\emph{uniform ergodicity} of the chain, that is, any initial
probability distribution in the state space converges to the
unique stationary measure exponentially fast.

Next result gives a sufficient condition to obtain an
aperiodic Harris chain satisfying Doeblin's Condition via
random perturbation of a dynamical system.

\begin{proposition}
  \label{prop:ran-per-aper-Harris}
  Let $f$ be a map randomly perturbed according to
  \eqref{eq:random-perturbation} with noise distribution
  given by \eqref{eq:noise-distribution} and such that there
  exists a full $\l$-measure subset $Y$ and a point
  $x_0\in\M$ satisfying $\{f^nx_0: n\ge0\}$ is both dense in
  $\M$ and a subset of $Y$. Then the random perturbation
  defines an aperiodic Harris chain which satisfies
  Doeblin's Condition.
\end{proposition}

\begin{proof}
  For the Harris conditions we need to find the sets
  $A,\, B$ as in Definition \ref{def:Harris-chain}.  Fix any
  positive Lebesgue measure Borel set $A$.  For condition
  (i) of Definition~\ref{def:Harris-chain}, we show that
  $\theta^\N\{\o: \exists n\ge1$ s.t.
  $f^n_{\o}(z)\in A\}>0$ for $\l$ a.e. $z\in\M$.

  By assumption on $f$, there exists a point $x_0\in Y$ so
  that $\overline{\{f^nx_0:n\geq 1\}}=\M$. Hence we can find
  $N\in\N$ so that $\{f^nx_0:0\le n\le N\}$ is
  $\frac{\rho_0}4$-dense and $\{f^nx_0: k\le n\le N+k\}$ is
  also $\frac{\rho_0}4$-dense for all $k=1,\dots,N$.

  Let $w\in A$ be a Lebesgue density point of $A$. Then, for
  $\l$-a.e.  $z\in\M$, we can find
  $x_1\in B_{\rho_0}(fz)\subset f^z\big(\supp(\theta)\big)$
  so that $x_1=f^kx_0$ for some $0\le k\le N$; and so there
  exists $k\le n \le N+k$ such that
  $f^nx_1\in B_{\rho_0/4}(w)$. In particular,
  $\l\big(B_{\rho_0}(f^nx_1)\cap A\big)>0$.

  Then $\dist(w,f^nx_1)<\frac{\rho_0}4$ and
  $\dist(fz,x_1)<\rho_0$. Since $x_1\in Y$ we have that
  $q_{f^kx_1}\ge\un{q}$ in a $\rho_0$-neighborhood of
  $f^{k+1}x_1$ for $k=0,\dots,n-1$. Consequently, by
  definition of $(q_x)_{x\in\M}$, for all small enough
  $\delta>0$
  \begin{align*}
    [(f^z)^{n+1}]_*\theta^\N\big(B_\delta(w)\big)
    &=
      \int \dif\theta^\N(\o) 1_{B_\delta(w)}\circ
      f^{n+1}_{\o}(z)
    \\
    &=
      \int \dif z_1\cdots \dif z_n \,
      q_z(z_1)q_{z_1}(z_2) q_{z_2}(z_3)\cdots
      q_{z_{n-1}}(z_n) 1_{B_\delta(w)}(z_n)
    \\
    &\ge
      \un{q}^{n}\cdot\l(B_\delta(w))>0
  \end{align*}
  In particular, we get
  $[(f^z)^{n+1}]_*\theta^\N\big(A\big)>0$ by the choice of
  $w$. This also shows that
  \begin{align}\label{eq:strongrecurrent}
        \l(A)>0\implies \tau_A(x)\le 2N, \quad\l-\text{a.e.}\, x\in\M
  \end{align}
  and proves the following stronger statement than item (i)
  of Definition~\ref{def:Harris-chain}
  \begin{align}\label{eq:strongHarris}
    \l(A)>0\implies\text{for $\l$-a.e. }z\in\M\, 
    \exists 0\le n(z)\le 2N:
    p_z^{n+1}(A)>\un{q}^n\l(A\cap B_\delta(w))
  \end{align}
  for every Lebesgue density point $w$ of $A$ and every
  small enough $\delta>0$.

  \begin{remark}
    \label{rmk:loconst}
    The function $n=n(z)$ is locally constant for
    $\l$-a.e.\ $z$ since $x_1\in B_{\rho_0}(f\wt{z})$ for all
    $\wt{z}$ in a neighborhood of $z$, by the continuity of
    $f$ on $Y$.
  \end{remark}

  To obtain item (ii) of Definition~\ref{def:Harris-chain},
  we take $B=B_{\rho_0/2}(f(x))$, for any given fixed
  $x\in A$.  Then, for any $C\subset B$, using properties
  ~\eqref{eq:noise-distribution} we get
\begin{align*}
p_xC=p(x, C)
&= 
[(f^x)_*\theta](C)
=
\int_C q_x\,\dif\l 
\geq
\un{q}\cdot \l(C)
=
\un{q}\cdot \l(B)\cdot \frac{\l{(C\cap B)}}{\l(B)}
\end{align*}
and so taking $\xi=\un{q}\cdot\l(B)$,
$m(C):=\frac{\l(C\cap B)}{\l(B)}$ for any Borel set
$C\subset B$ and $g(x,y):=q_x(y)$ we are done.

Aperiodicity is a consequence of properties
\eqref{eq:noise-distribution} together with the existence of
a dense unperturbed orbit. Let us
assume that there is a partition of $\M$ as in
Definition~\ref{def:aperiodicity}. We can assume without
loss of generality that $\l(\M_i)>0$ for all
$i=1,\dots,\ell$. Let $\tilde\M_i$ denote the subset of
Lebesgue density points of $\M_i$ so that
$\l(\M_i\setminus\tilde\M_i)=0$ and $i=1,\dots,\ell$ and thus
$\M=\sqcup_{i=1}^\ell \tilde\M_i,\l\bmod0$.  In particular
we obtain $\M=\cup_{i=1}^\ell\ov{\tilde\M_i}$ and this
cannot be a disjoint union, for otherwise
$\sqcup_{i=1}^{\ell-1}\ov{\tilde\M_i}=\M\setminus\ov{\tilde\M_\ell}$
and $\ov{\tilde\M_\ell}$ would be open and closed,
contradicting the connectedness of $\M$ because $\ell\ge2$.

Hence there exists $x\in\ov{\tilde\M_i}\cap\ov{\tilde\M_j}$
for some $i\neq j$ and we can find $y=f^nx_0\in\M$ for some
$n\ge1$ such that $B_{\rho_0/2}(fy)\ni x$ (recall that $x_0$
has dense positive orbit). Then we obtain
\begin{align*}
[(f^y)_{*}\theta](\tilde\M_i)
&\ge
\int_{B_{\rho_0}(fy)\cap\tilde\M_i}q_y\,\dif\l
\ge
\un{q}\cdot\l\big(B_{\rho_0}(fy)\cap\tilde\M_i\big)>0
\end{align*}
since $B_{\rho_0}(fy)\cap\tilde\M_i\neq\emptyset$, by
definition of Lebesgue density point. Analogously we get
$[(f^y)_{*}\theta](\tilde\M_j)>0$. Therefore we have found
$y\in\M$ such that $0<p_y\M_i<1$ and $0<p_y\M_j<1$. This
shows that a partition of $\M$ as in
Definition~\ref{def:aperiodicity} cannot exist.

To obtain Doeblin's Condition we use
properties~\eqref{eq:noise-distribution}. Let
$\gamma\in(0,1)$ be such that $\l(B(z,\rho_0))\ge\gamma$ for
all $x\in\M$, which exists by compactness of $\M$. Let $B$
be a Borel subset of $\M$ such that $\l(B)\ge1-\gamma/2$.
Then for any Borel subset $A$
\begin{align*}
\l(A)>\gamma
\implies
p_xB=\int_Bq_x\,\dif\l
\ge
\un{q}\l(B\cap B(fx,\rho))
\ge
\un{q}\frac{\gamma}2, \quad\l-\text{a.e. } x\in\M.
\end{align*}
Hence letting $\delta=\un{q}\gamma/2$ we have obtained
Doeblin's Condition for $k=1$.
\end{proof}


\subsection{Strictly positive stationary density}
\label{sec:stationary-density}

We note that the absolute continuity assumption on
$f^x_*\theta\ll\l$ for $\l$-a.e.\ $x\in\M$ ensures that the
unique stationary probability measure is given by a
distribution, that is, $\p=\mu\times\theta^\N$ where
$\mu=h\l$ with $h\ge0, h\in L^1(\l)$. Indeed, by definition
of stationary measure, for $\phi\in L^1(\mu)$
\begin{align*}
 \mu(\phi)
  &=
  \int \dif\mu(x) \int \dif\theta(\omega) \, \phi\circ f_\omega(x)
  =
  \int \dif\mu(x) \int \dif(f^x_*\theta) \, \phi
  =
  \int \dif\mu(x) \int \dif\l \, \phi\cdot q_x
\end{align*}
and if $\phi=1_A$ with $A\in\B$ so that $\l(A)=0$, we get
$\mu(A)=\int \dif\mu(x) \int_A \dif\l \, q_x = 0$, showing that
$\mu\ll\l$ and that $h=\frac{\dif\mu}{\dif\l}$ is as claimed.

This density is strictly positive as a consequence of
\eqref{eq:strongrecurrent} and \eqref{eq:strongHarris}
together with Remark~\ref{rmk:loconst}.
Indeed, because $\mu$ is stationary we get, fixing $z\in Y$
and  a Borel
subset $A$ such that $\l(A)>0$, the existence of $0\le n\le
2N$ satisfying
\begin{align*}
  \mu(A_\delta^w)
  &=
  \int \dif(\mu\times\theta^\N)(x,\o) 1_{A_\delta^w}\circ f^n_{\o}(x)
  =
  \int \dif\mu(x) \dif [(f^x)^n]_*\theta^\N\, 1_{A_\delta^w}
    \ge \un{q}^{n-1}\mu(B_{\bar\delta}(z))\l(A_\delta^w)
\end{align*}
where $A_\delta^w:=A\cap B_\delta(w)$, for every Lebesgue
density point $w$ of $A$ and every small enough
$\delta,\bar\delta>0$. Note that $\delta$ depends on $w$ and
$A$ and $\bar\delta$ depends on $z$, but $n=n(z)$ is
uniformly bounded from above. Since $\dif\mu=h\, \dif\l$ we obtain
for all small enough $\delta>0$
\begin{align*}
  \frac1{\l(B_\delta(w))}\int h1_{A_\delta^w}\,\dif\l
  \ge
  \un{q}^{n-1}\mu(B_{\bar\delta}(z))
  \frac{\l(A\cap B_\delta(w))}{\l(B_\delta(w))}
\end{align*}
and by the Lebesgue Differentiation Theorem and the choice
of $w$ we obtain
\begin{align*}
  h(w)\ge\mu(B_{\bar\delta}(z))\cdot\inf\{\ov{q}^k:0\le k\le2N\}=:\un{h}
\end{align*}
Hence we conclude
$h\ge\un{h}>0, \l$-a.e..


\subsection{Deduction of the main results}
\label{sec:proof-main-results}

Now we use the previous observations to complete the proofs
of the main results.

\begin{proof}[Proof of Theorems~\ref{mthm:trans->DC}
and \ref{mthm:random-pert-decay}]
Since the chain is an aperiodic Harris chain which also
satisfies Doeblin's condition (see
Proposition~\ref{prop:ran-per-aper-Harris}), we can use the
equivalence of items $(ii)$ and $(iv)$ in \cite[Theorem
16.0.2]{MT09} and conclude uniform ergodicity: there exist
$\lambda>1$ and $C<\infty$ such that for $\l$-a.e. $x\in\M$
and each $n\ge1$
\begin{equation}
\label{eq:conv-to-equi}
\|p^n_x-\mu\|\leq C \lambda^{-n},
\end{equation}
where $\|\cdot\|$ stands for the total variation norm, \ie
$\|p^n_x-\mu\|=\sup_{A\in\B}|p^n_xA-\mu(A)|$. Since all
probability measures involved here have densities,
\eqref{eq:conv-to-equi} is equivalent to 
\begin{equation}
\label{eq:tot-var-2}
\frac{1}{2}\left\|q_x^n-h\right\|_{L^1(\l)}
=
\frac{1}{2}\int\left |q_x^n-h\right|\,\dif\l
\le 
C\lambda^{-n}.
\end{equation}
If $\psi\in L^\infty(\l)=L^\infty(\mu)$ and $\phi\in L^1(\mu)$, then
\begin{align}
  \int\psi\circ f^n_\omega(x)\cdot\phi(x)\, \dif\p(x,\omega)
  &-
  \mu(\psi)\mu(\phi)
  =
  \int \dif\mu(x) \, \phi(x)\int \dif[(f^x)^n_*\theta^\N] \psi - \int
    \dif\mu\,\phi\cdot\int \dif\mu\, \psi \nonumber
  \\
  &=
  \int \dif\mu(x) \phi(x)\left(
    \int \dif\l \, \psi \cdot (q_x^n-h)
    \right) \quad\text{(using \eqref{eq:tot-var-2})}\nonumber
  \\
  &\le
  \int \dif\mu\, \phi\cdot \|\psi\|_\infty\cdot 2C\lambda^{-n}
  \le \label{eq:expdecay}
  2C\|\psi\|_{L^\infty(\mu)}\|\phi\|_{L^1(\mu)}\lambda^{-n}
\end{align}
concluding the proof of annealed decay of correlations
agains $L^1$ observables.
\end{proof}

\begin{proof}[Proof of Theorem~\ref{mthm:random-EVL-finite}
  and Corollary~\ref{cor:random-EVL=>Poisson}]

  As already explained in
  Section~\ref{sec:statement-results}, it is enough to show
  that \eqref{eq:expdecay} implies all conditions
  $D_2(u_n), D_3(u_n)$ and $D'(u_n)$.

  Condition $D_2(u_n)$ is designed to follow easily from
  fast decay of correlations. In fact, recalling
  (\ref{def:Un}), if we choose $\phi=1_{U_n}$ and
  $\psi=\int 1_{\{\vfi(x),\vfi\circ
    f_{\tilde{\omega}_1}(x),\ldots,\vfi\circ
    f_{\tilde{\omega}_{\ell-1}}(x)\leq u_n\}}
  \dif\theta^{\ell-1}(\tilde{\o})$,
  then we can take $\gamma(n,t)=\gamma(t)=C^*\lambda^{-t}$
  in Condition $D_2(u_n)$ for some $C^*>0$, $\lambda>1$ and
  $t_n=o(n)$ coming from \eqref{eq:expdecay}.  A very
  similar reasoning applies to get Condition $D_3(u_n)$ by
  choosing the same $\phi$ but
  $\psi=\int 1_Z\dif\theta^{\N}(\tilde{\o})$;
  where $A\in\RR$ and $Z=Z(\tilde{\o})=\bigcap_{i\in A\cap
      \N}{\{x:f^i_{\tilde{\o}}(x)\leq
      u_n\}}$; see the proof of \cite[Theorem D and
  Corollary F]{AFV15}.
  
  Now we show that $D'(u_n)$ holds automatically in our
  random setting.  By assumption~\ref{item:U-ball} we have
  that, for $n$ sufficiently large, $U_n$ is topologically a
  ball around a chosen point $\zeta$ and
  $n\mu(U_n)\xrightarrow[n\to\infty]{}\tau$ by
  \eqref{eq:un}. Moreover, we have from
  (\ref{eq:noise-distribution}) that
  \begin{align*}
   (f^x_*\theta)U_n=\int
    q_x\cdot1_{U_n}\,\dif\l\le\ov{q}\cdot\l(U_n), \quad
    \l-\text{a.e. } x\in\M.
  \end{align*}
  Consequently, we can estimate
  \begin{align*}
    \p\big(\{(x,\o)\in\Omega: x\in U_n &\qand f^j_{\o}(x)\in U_n\}\big)
    =
    \int_{U_n} \dif\mu(x)1_{U_n}(x)\int \dif\theta^\N(\o) 
    1_{U_n}\circ f^j_{\o}(x)
    \\
    &\le
      \mu(U_n)\int \dif\mu(x)\int \dif\theta^{j-1}(\o) 
      [(f^{f^{j-1}_{\o}x})_*\theta](U_n)
    \\
    &\le
      \mu(U_n)\cdot\ov{q}\cdot\l(U_n)
      \le
      \frac{\ov{q}}{\un{h}}\cdot\mu(U_n)^2
  \end{align*}
  since $\mu\ge\un{h}\,\l$ by
  Subsection~\ref{sec:stationary-density}. Hence we get
  \begin{align*}
    n\sum_{j=1}^{\lfloor n/k_n \rfloor}
    \p( X_0>u_n , X_j> u_n)
    &=
    n\sum_{j=1}^{\lfloor n/k_n \rfloor}
    \p\big(\{(x,\o)\in\Omega: x\in U_n \qand f^j_{\o}(x)\in U_n\}\big)
    \\
   &\leq
     n\left\lfloor\frac{n}{k_n}\right\rfloor\frac{\ov{q}}{\un{h}}\mu(U_n)^2
     \le
     \frac{\ov{q}}{\un{h}}\cdot \big(n\mu(U_n)\big)^2
     \cdot\frac1{k_n}
     \xrightarrow[n\to\infty]{}0
  \end{align*}
  since $k_n\to\infty$ by
  definition~\eqref{eq:kn-sequence-1}. This completes the
  proof of Condition $D'(u_n)$.
\end{proof}


\def\cprime{$'$}

\bibliographystyle{abbrv}

\end{document}